\documentclass{amsart}
\usepackage{graphicx}
\newtheorem{thm}{Theorem}[section]

\newtheorem{lem}[thm]{Lemma}
\newtheorem{prop}[thm]{Proposition}
\theoremstyle{definition}

\theoremstyle{remark}
\newtheorem{rem}[thm]{Remark}
\numberwithin{equation}{section}
\begin{document}

\title[Levi-Civita functional equation ]{A note on  Levi-Civita functional equation}%
\author{Belfakih Keltouma and Elqorachi Elhoucien}%
\address{}%
\email{}%


%\date{}%
%\dedicatory{}%
%\commby{}%
% ----------------------------------------------------------------
\begin{abstract}
In this paper we find the solutions of the functional equation
$$f(xy)=g(x)h(y)+\sum_{j=1}^{n}g_j(x)h_j(y),\;x,y \in
M,$$ where $M$ is a monoid,  $n\geq 2$, and $g_j$ (for $j=1,...,n$)
are linear combinations of at least $2$ distinct nonzero
multiplicative functions.
\end{abstract}
\maketitle
% ----------------------------------------------------------------
\section{introduction }
A solution of Levi-Civita functional equation on a monoid $M$ is an
ordered set of functions $f,g_1,...,g_N,h_1,...,h_N$$
:M\longrightarrow \mathbb{C} $ satisfying Levi-Civita functional
equation
\begin{equation} \label{Levi} f(xy)=\sum_{j=1}^{N}g_j(x)h_j(y),\; x,y \in M.\end{equation}
Shulman \cite[Lemma 4]{171} described the measurable solutions of
(\ref{Levi}) on locally compact groups in terms of group
representations under the additional assumption that both
$g_1,...,g_N$ and $h_1,...,h_N$ are linearly independent. A
description of the (nondegeneated) solutions on abelian groups can
be found in Sz�kelyhidi \cite{s88} for all $N$. They turn out to
be exponential polynomials, that is, sums of products of polynomials
of additive functions and of solutions of
$\varphi(x+y)=\varphi(x)\varphi(y)$. That is in general not so for
non-abelian groups as Stetk{\ae}r
\cite[Exemple 1]{st} reveals. We refer also to \cite[Theorem 5.2]{book}.\\
Chung, Kannappan and Ng \cite{45} solved the Levi-Civita functional
equation
\begin{equation} \label{Levi1} f(xy)=f(x)g(y)+f(y)g(x)+h(x)h(y),\; x,y \in
G,\end{equation} where $G$ is a group.\\ The functional equation
(\ref{Levi}) was for each $N\leq 4$ thoroughly worked out on abelian
groups with implicit formulas for the solutions by Sz�kelyhidi
[\cite{197}, Theorem 10.4]. Ebanks \cite{EB} studied the functional
equation
\begin{equation} \label{Eb} f(xy)=k(x)L(y)+g(x)h(y),\;
x,y \in S,\end{equation} for four unknown central functions
$f,g,h,k$ on certain non abelian semigroups $S$, where $L$ is a
fixed multiplicative function on S. Stetk{\ae}r \cite{st} removed
the restriction that the solutions of (\ref{Eb}) should be central
and solved the functional equation \begin{equation} \label{stc}
f(xy)=g_1(x)h_1(y)+\mu(x)h_2(y),\;x,y \in G,\end{equation} where $G$
is a group (that need not to be abelian), and $\mu:G\longrightarrow
\mathbb{C}$ is a character of $G$.\\The motivation of the present
work is the recent paper, by Ebank and Stetk{\ae}r \cite{st2} in
which they derive explicit formulas for the solutions
$f,g_1,h_1,h_2:M\longrightarrow K$ of the functional equation
\begin{equation}\label{299} f(xy)=g_1(x)h_1(y)+g(x)h_2(y),\; x,y \in
M,\end{equation}  where $M$ is a monoid (that need not to be
abelian), $K$ is a field, $g$ is a linear combination of $n\geq 2$
distinct nonzero multiplicative functions with nonzero
coefficients.\\
We wish to see what makes the paper \cite{st2} work by analyzing a
more general Levi-Civita functional equation
\begin{equation} \label{gmn}f(xy)=g(x)h(y)+\sum_{j=1}^{n}g_j(x)h_j(y),\;x,y \in M, \end{equation} where $M$
is a monoid, $K$ is a field, and $f,g,h,h_j:M\longrightarrow K$ are
the unknown functions, and where for all $j\in \{1,..,n\}$, $g_j$ is
defined by \begin{equation}
\label{fgn}g_j=\sum_{i=n_j}^{m_j}b_i\mu_i,\end{equation} where
$n_1=1$, $m_j=n_{j+1}-1$ for $j\in \{1,2,...,n-1\}$, $m:=m_{n}$,
$m_{j}-n_j \geq 1 $, for $j\in \{1,2,...,n\}$,  $b_i \neq 0$ for all
$i\in \{1,...,m\}$, and $\{\mu_1,\mu_2,...,\mu_m\}$ are distinct
nonzero multiplicative functions on $M$. Thus we have grouped the
elements of the sequence $\mu_1,\mu_2,...,\mu_m$ into $n$
consecutive, disjoint intervals by
$$\mu_{n_1}=\mu_1,...,\mu_{m_1};\mu_{n_2},...,\mu_{m_2};\mu_{n_j},...,\mu_{m_j};...;\mu_{n_n},...,\mu_{m_n}=\mu_m$$
in accordance with the sequence $g_1;g_2;...;g_j;...;g_n$. We note
here that $g_j$ has at least $2$ terms on the right hand side like
in the motivating paper \cite{st2}.\\
Equation (\ref{gmn}) is an extension of the functional equation
(\ref{299}) (Take $n=2$ and $h_2=0$ in (\ref{gmn})).\\
In our present paper:
\begin{enumerate}\item We get explicit solution formulas involving
multiplicative functions.\item The solutions are abelian, except
when they blatantly need not be abelian, so that non-abelian
phenomena essentially do not occur, except for some arbitrary
functions. This contrasts the formally simpler functional equation
studied in the paper \cite{st} (It is due to $m_j-n_j\geq 1$).\item
The proofs are elementary algebraic manipulations (no homological
algebra and the link). No analysis or geometry come into play.
\end{enumerate}
The key elements of our set up are
\begin{enumerate} \item The functions are defined on a monoid
$M$ with an identity element $e$, and map into a field $K$.
\item Whenever we refer to the functional equation (\ref{gmn}, then the functions $g_j$, $j=1,2,...n$ have the
forms (\ref{fgn}).
\item The multiplicative functions $\mu_i$, appearing in (\ref{gmn}) are nonzero
and distinct and their coefficients $b_i$ are nonzero.
\end{enumerate}

 Our results
are organized as follows. We discuss two cases according to whether
$f\neq 0$ and the set $\{g_1,\mu_1,\mu_2,...,\mu_m\}$ is linearly
independent
 (Proposition
\ref{indpnh}) or not (Proposition \ref{depn}). The main results are
given in Theorem \ref{soln}.\\
Throughout this paper $M$ denotes a monoid: A semigroup $S$ with a
neutral element $e$, and $K$ is a field. Let $K^\ast =K\setminus
\{0\}$ denote the subset of nonzero element. A multiplicative
function $\mu:M\longrightarrow K$ is a function such that
$\mu(xy)=\mu(x)\mu(y)$ for all $x,y \in M$. Let $M(n,\mathbb{C})$
denote the algebra of all complex $n\times n$ matrices, and let
$A^{T}$ denote the transpose matrix of a matrix $A \in
M(n,\mathbb{C})$, and $A^{-1}$ its inverse. Let $n \in \mathbb{N}$
such that $n\geq 2$, and let $q\in \{1,...,n\}$, we introduce the
following notation
$$\{1,2,...,\widehat{q},...,n\}:=\{1,...,n\} \backslash\{q\}.$$ The
fact that distinct multiplicative functions are independent is
frequently used in the present paper. The following result is taken
from Proposition 1(a) in \cite{st2}.
\begin{prop} \label{mudep}Let $n\in \mathbb{N}$, let $\mu_1,\mu_2,...,\mu_n:S\longrightarrow
K$ be $n$ distinct multiplicative functions, and let
$a_1,a_2,...,a_n \in K$.\\ If $a_1\mu_1+a_2\mu_2+...+a_n\mu_n=0$,
then $a_1\mu_1=a_2\mu_2=...=a_n\mu_n=0$. Thus any set of distinct
nonzero multiplicative functions is linearly independent.
\end{prop}
\begin{rem} \label{rmk} Any $i\in \{1,2,...,m \}$ in the form (\ref{fgn}), belongs to
exactly one of the $n$ disjoint subintervals $[n_j,m_j]$ of $[1,m]$,
and so the number $j\in [1,n]$ of that subinterval is uniquely
determined by $i$. Let $J(i):=j\in [1,n]$, when $j$ is the number of
the subinterval containing $i$. So $i \longmapsto J(i)$,
$i=1,2,...,m$, is a function of $i$. Note that $J:
[1,m]\longrightarrow [1,n]$ is surjective, that $i\in
[n_{J(i)},m_{J(i)}]$ for $i=1,2,...,m$, and that
$J^{-1}(\{k\})=[n_k,m_k]$ for all $k\in \{1,2,...,n\}.$\\Hence, we
have
\begin{equation}\label{frmij} \sum_{j=1}^{n}g_j(x)h_j(y)=\sum_{j=1}^{n}\sum_{i=n_j}^{m_j}b_i\mu_i(x)h_j(y)=\sum_{i=1}^{m}b_i\mu_i(x)h_{J(i)}(y).\end{equation}
Thus, the functional equation (\ref{gmn}) has another form
\begin{equation} \label{nf}f(xy)=g(x)h(y)+\sum_{i=1}^{m}b_i\mu_i(x)h_{J(i)}(y),\; x,y \in
G.\end{equation}
\end{rem}
\section{General solution of the functional equation (\ref{gmn})}
The following proposition gives a partial solution of the functional
equation (\ref{gmn}).
\begin{prop} \label{indpn} Let $(f,g,h,h_1,h_2,...,h_n)$ be a solution of the functional
equation (\ref{gmn}) such that $f\neq 0$, $h(e)\neq 0$, and such
that the set $\{g,\mu_1,\mu_2,...,\mu_m\}$ is linearly independent.
Then there exist a nonzero multiplicative function $\chi:M
\longrightarrow K$ such that $\chi \neq \mu_i$ for all $i \in
\{1,2,...,m\}$, and $a,b \in K^{\ast},$ and $c_j \in K $ with $ j\in
\{1,2,...,n\}$ such that
$$f=ab\chi, \;g=b\chi -\sum_{k=1}^{n}c_kg_k,\;
h=a\chi,\;h_j=ac_j\chi,$$ where $ j\in \{1,2,...,n\}$.
\end{prop}
\begin{proof} Taking $y=e$ in (\ref{gmn}) we get \begin{equation} \label{wx0n}
f=h(e)g+\sum_{j=1}^{n}h_j(e)g_j.\end{equation} Substituting
(\ref{wx0n}) in equation (\ref{gmn}) we get
$$h(e)g(xy)+\sum_{j=1}^{n}h_j(e)g_j(xy)=g(x)h(y)+\sum_{j=1}^{n}g_j(x)h_j(y),$$
from which we deduce that
\begin{equation} \label{wx1n}
g(xy)=\frac{1}{h(e)}[g(x)h(y)+\sum_{j=1}^{n}g_j(x)h_j(y)-\sum_{j=1}^{n}h_j(e)g_j(xy)].\end{equation}
Let $x,y,z$ be in $M$, using (\ref{gmn}) we have
$$f((xy)z)=g(xy)h(z)+\sum_{j=1}^{n}g_j(xy)h_j(z),$$
$$f(x(yz))=g(x)h(yz)+\sum_{j=1}^{n}g_j(x)h_j(yz).$$
Using the associativity of the monoid operation, we deduce that
$$g(xy)h(z)+\sum_{j=1}^{n}g_j(xy)h_j(z)=g(x)h(yz)+\sum_{j=1}^{n}g_j(x)h_j(yz).$$
Replacing $g(xy)$ in this equation by its form in (\ref{wx1n}) we
obtain after computations
$$g(x) \left[h(y)\frac{h(z)}{h(e)}-h(yz)\right]=\sum_{j=1}^{n}g_j(x)\left[h_j(yz)-h_j(y)\frac{h(z)}{h(e)}\right]+
\sum_{j=1}^{n}g_j(xy)\left[h_j(e)\frac{h(z)}{h(e)}-h_j(z)\right]$$
$$=\sum_{j=1}^{n}\left(\sum_{i=n_j}^{m_j}b_i\mu_i(x)\left[h_j(yz)-h_j(y)\frac{h(z)}{h(e)}\right]
+\sum_{i=n_j}^{m_j}b_i\mu_i(xy)\left[h_j(e)\frac{h(z)}{h(e)}-h_j(z)\right]\right)$$
$$=\sum_{j=1}^{n} \sum_{i=n_j}^{m_j}b_i\mu_i(x)\left[h_j(yz)-h_j(y)\frac{h(z)}{h(e)}+\mu_i(y)\left(h_j(e)\frac{h(z)}{h(e)}-h_j(z)\right)\right].$$
$$=\sum_{i=1}^{m} b_i\mu_i(x)\left[h_{J(i)}(yz)-h_{J(i)}(y)\frac{h(z)}{h(e)}+\mu_i(y)\left(h_{J(i)}(e)\frac{h(z)}{h(e)}-h_{J(i)}(z)\right)\right].$$
Since by hypothesis, $\{g,\mu_1,\mu_2,...,\mu_m\}$ is linearly
independent and all $b_i$ are nonzero we deduce that
 \begin{equation} \label{wx00n}
h(yz)-h(y)\frac{h(z)}{h(e)}=0,\end{equation}and
\begin{equation} \label{wx01n}
h_{J(i)}(yz)-h_{J(i)}(y)\frac{h(z)}{h(e)}+\mu_i(y)(h_{J(i)}(e)\frac{h(z)}{h(e)}-h_{J(i)}(z))=0,\end{equation}
for all $i=1,2,...,m.$\\
Equation (\ref{wx00n}) implies that $h=a \chi$, where $\chi:=h/h(e):
M\longrightarrow K$ is a nonzero multiplicative function and
$a=h(e)\neq 0.$ Using the facts that $J^{-1}(\{j\})=[n_j,m_j]$,
$\mu_{n_j},...,\mu_{m_j}$ are distinct, and $m_j-n_j\geq 1$ for all
$j\in \{1,2,...,n\},$ we get from (\ref{wx01n}) that
$$h_{J(i)}(z)=h_{J(i)}(e)\frac{h(z)}{h(e)}=ac_{J(i)}\chi(z), \text{ where }
c_{J(i)}:=\frac{h_{J(i)}(e)}{a}.$$ Since the function $J$ is
surjective we deduce that for all $j\in \{1,2,...,n\}$
$$h_j(z)=h_j(e)\frac{h(z)}{h(e)}=ac_j\chi(z), \text{
where } c_j:=\frac{h_j(e)}{a}.$$ Taking $x=e$ in (\ref{wx1n}) we
find that
\begin{equation} \label{g1n} g=b\chi -\sum_{k=1}^{n}c_kg_k, \text{ where }
b:=g(e)+\sum_{k=1}^{n}c_kg_k(e) \in K.
\end{equation} If $b=0$ we get that
$$g=-\sum_{k=1}^{n}c_kg_k=-\sum_{i=1}^{m}c_{J(i)}b_i\mu_i.$$This means
that $\{g,\mu_1,...,\mu_m\}$ is a linear dependent set, which is not
true by hypothesis. Thus, $b\neq 0$. Using the same argument we
prove that $\chi $ is distinct from
$\mu_1,\mu_2,...,\mu_m$.\\
To find $f$ we replace the functions $g,h,h_1,...,h_n$ by their
forms in (\ref{wx0n}) and we obtain that $f=ab\chi$.
\end{proof}

The following lemmas will be useful in the proofs of our main
results.
\begin{lem} \label{f0n} Let $(f,h_1,h_2,...,h_n)$ be a solution of the functional equation
\begin{equation} \label{gG0n}f(xy)=\sum_{j=1}^{n}g_j(x)h_j(y),\;x,y\in
M.
\end{equation} Then
$f=h_j=0$ for all $j\in \{1,...,n\}.$
\end{lem}
\begin{proof}Assume for contradiction that $f\neq 0$. Taking $y=e$ in (\ref{gG0n}) we get that
\begin{equation} \label{h(e)}f=\sum_{j=1}^{n}h_j(e)g_j.\end{equation} If $h_j(e)=0$ for all $j \in \{1,...,n\}$
then we get from (\ref{h(e)}) that $f=0$ which contradicts the
hypothesis that $f\neq 0$. Thus, there exists $k\in \{1,...,n\}$
such that $h_k(e)\neq 0$. We may take $k=n$, and the equation
(\ref{gG0n}) becomes
$$f(xy)=g_n(x)h_n(y)+\sum_{j=1}^{n-1}g_j(x)h_j(y).$$ The set
$\{g_n,\mu_1,...,\mu_{m_{n-1}} \}$ is linearly independent. In fact,
if it is not, we will have
\begin{equation} \label{libre}\sum_{i=1}^{m_{(n-1)}}c_i\mu_i(x)+c_ng_n=0\end{equation} with $(c_1,c_2,...,c_n) \neq
(0,0,...,0)$. Using the form of $g_n$, equation (\ref{libre}) can be
written as follows
$$\sum_{i=1}^{m_{(n-1)}}c_i\mu_i(x)+\sum_{i=n_n}^{m}c_nb_i\mu_i(x)=0.$$
Since $n_n=m_{n-1}+1$ the last equation is equivalent to
$\sum_{i=1}^{m}a_i\mu_i(x)=0$, where $a_i=c_i$ for
$i=1,2,...,m_{n-1}$, and $a_i=c_nb_i$ for $i=n_n,...,m$. The
multiplicative functions $\mu_1,\mu_2,...,\mu_m$ are distinct and
nonzero, so according to Proposition \ref{mudep} we deduce that
$a_i=0$ for all $i\in \{1,2,...,m\}$, and hence $c_i=0$ for
$i=1,2,...,m_{n-1}$, and $c_n=0$. This means that
$(c_1,c_2,...,c_n)=(0,0,...,0)$ and this contradicts the hypothesis
that the set $\{g_n,\mu_1,...,\mu_{m_{n-1}}\}$ is linearly
dependent. Thus $\{g_n,\mu_1,...,\mu_{m_{n-1}}\}$ is a linearly
independent set. Now, using the fact that $f\neq 0$, $h_{n}(e) \neq
0$, and that $\{g_n,\mu_1,...,\mu_{m_{n-1}}\}$ is linearly
independent, we get from Proposition \ref{indpn} that there exist a
nonzero multiplicative function $\chi:M \longrightarrow K$ such that
$\chi \neq \mu_i$ for all $i \in \{1,2,...,m_{n-1}\}$, and $a,b \in
K^{\ast}$, and $c_j \in K,$ for $j=1,2,...,n$ such that
\begin{equation} \label{solln}f=ab\chi, \;g_{n}=b\chi
-\sum_{j=1}^{n-1}c_jg_j,\; h_n=a\chi,\;h_j=ac_j\chi \end{equation}
for all $j=1,2,...,n-1$. We get from (\ref{solln}) that
$$b\chi=\sum_{j=1}^{n-1}c_jg_j+g_{n}=\sum_{i=1}^{m_{n-1}}c_{J(i)}b_i\mu_i+\sum_{i=n_n}^{m}b_i\mu _i.$$
This means that the set $\{\chi,\mu_1,\mu_2,...,\mu_m\}$ is linearly
dependent
 which is not possible according to Proposition \ref{mudep} because $\chi,\mu_1,\mu_2,...,\mu_m$ are distinct multiplicative functions.
 We deduce that $f=0$. Hence, we get from the functional equation
 (\ref{gG0n}) that
$$\sum_{j=1}^{n}g_j(x)h_j(y)=\sum_{i=1}^{m}b_i(x)h_{J(i)}(y)=0.$$ Since $\{\mu_1,...,\mu_m\}$ is linearly independent we deduce that
$b_ih_{J(i)}=0$. Using the fact that the coefficients $b_i$ are
nonzero we deduce that $h_{J(i)}=0$. The subjectivity of the
function $J$ implies that $h_j=h_{J(i)}=0$ for all $j\in
\{1,...,n\}.$ This completes the proof.
\end{proof}
\begin{lem} \label{depnn} Let $(f,g,h,h_1,...,h_n)$ be  a solution of the functional
equation (\ref{gmn}) such that $f\neq 0$. If the set
$\{g,\mu_1,\mu_2,...,\mu_m\}$ is linearly independent, then the set
$\{f,\mu_1,\mu_2,...,\mu_m\}$ is also linearly independent.
\end{lem}
\begin{proof} If $h=0$ then we have
$$f(xy)=\sum_{j=1}^{n}g_j(x)h_j(y).$$ Applying Lemma \ref{f0n} we deduce
that $f=0$, which contradicts the hypothesis that $f\neq 0$. Hence
$h\neq 0$. Now assume that $\{f,\mu_1,...,\mu_m\}$ is linearly
dependent. This means that $f$ can be written as follows
$$f=\sum_{i=1}^{m}\alpha_i\mu_i,$$ where $\alpha_i\in K$ for all $i \in
\{1,...,m\}.$  Substituting the last expression of $f$ in equation
(\ref{nf}) we get \begin{equation}
\label{nfnn}\sum_{i=1}^{m}\alpha_i\mu_i(x)\mu_i(y)=g(x)h(y)+\sum_{i=1}^{m}b_i\mu_i(x)h_{J(i)}(y),\;
x,y \in G.\end{equation} Since $h\neq 0$ there exists $y_0$ such
that $h(y_0)\neq 0$. Thus, we obtain
$$g(x)=\sum_{i=1}^{m}\frac{\alpha_i\mu_i(y_0)-b_ih_{J(i)}(y_0)}{h(y_0)}\mu_i(x).$$
This means that $\{g,\mu_1,\mu_2,...,\mu_m \}$ is a linearly
dependent set. This completes the proof.
\end{proof}
Now, using Lemma \ref{f0n} and Lemma \ref{depnn} we can omit the
condition $h(e)\neq 0$ from Proposition \ref{indpn}.
\begin{prop} \label{indpnh} Let $(f,g,h,h_1,h_2,...,h_n)$ be a solution of the functional
equation (\ref{gmn}) such that $f\neq 0$. If the set
$\{g,\mu_1,\mu_2,...,\mu_m\}$ is linearly independent, then there
exist a nonzero multiplicative function $\chi:M \longrightarrow K$
such that $\chi \neq \mu_i$ for all $i \in \{1,2,...,m\}$, and $a,b
\in K^{\ast}$, and $c_j \in K$ for $j=1,2 ,...,n$, such that
\begin{equation} \label{+} f=ab\chi, \;g=b\chi
-\sum_{k=1}^{n}c_kg_k,\; h=a\chi,\;h_j=ac_j\chi \text{ for }
j=1,2,...,n.\end{equation} Conversely, the formulas (\ref{+}) define
solutions of equation (\ref{gmn}) .
\end{prop}
\begin{proof}The case of $h(e)\neq 0$ is Proposition
\ref{indpn}. Suppose that $h(e)=0$. The functional equation
(\ref{wx0n}) becomes
$$f=\sum_{j=1}^{n}h_j(e)g_j=\sum_{i=1}^{m}b_ih_{J(i)}(e)\mu_i.$$
This means that $\{f,\mu_1,...,\mu_m\}$ is linearly dependent.
According to Lemma \ref{depnn} we deduce that
$\{g,\mu_1,...,\mu_m\}$ is linearly dependent. This contradicts the
hypothesis that $\{g,\mu_1,...,\mu_m\}$ is linearly independent.
Thus $h(e)\neq 0$.\\ The converse can be proved easily by
substituting the formulas (\ref{+}) in (\ref{gmn}).
\end{proof}
Proposition \ref{depn} is the main result concerning the solutions
of (\ref{gmn}), when $f\neq 0$ and $\{g, \mu_1,...,\mu_m\}$ is
linearly dependent.
\begin{prop}\label{depn}
Let $(f,g,h,h_1,...,h_n)$ be a solution of the functional equation
(\ref{gmn}) such that $f\neq 0$. If the set $\{g,
\mu_1,\mu_2,...,\mu_m\}$ is linearly dependent, then there exist a
unique $q\in \{1,..,n\}$ and constants $a,b,c,d,c_j\in K$ for $j\in
\{1,..,\widehat{q},...,n\}$, with $c\neq d$ and $(a,b)\neq (0,0)$
such that  \begin{itemize} \item If $m_q-n_q=1$, then
\begin{equation}f=(d-c)(bb_{n_q}\mu_{n_q}-ab_{m_q}\mu_{m_q}),\end{equation}
\begin{equation}g=db_{n_q}\mu_{n_q}+cb_{m_q}\mu_{m_q}+\sum_{j=1,j\neq q}^{n}c_j
g_j, \end{equation} \begin{equation}h=a\mu_{m_q}+b\mu_{n_q},\;
h_q=-da\mu_{m_q}-cb\mu_{n_q},\end{equation}\begin{equation}h_j=-c_ja\mu_{m_q}-c_jb\mu_{n_q}
\text{ for all } j \in
\{1,..,\widehat{q},...,n\}.\end{equation}\item If $m_q-n_q \geq2$,
then \begin{equation}f=a(c-d)b_{k}\mu_{k},
\;g=dg_{q}+(c-d)b_{k}\mu_{k}+\sum_{j=1,j\neq q}^{n}c_j
g_j,\end{equation}
\begin{equation}h=a\mu_{k},\;h_{q}=-da\mu_{k},\;h_j=-c_j\mu_{k} \text{ for all
}j\in \{1,..,\widehat{q},...,n\},\end{equation} where $k\in
\{n_q,...,m_q\}$ and $a\neq 0$.\end{itemize}
\end{prop}
\begin{proof} That the set $\{g,
\mu_1,\mu_2,...,\mu_m\}$ is linearly dependent means that
$$\beta_0g+\sum_{i=1}^{m}\beta_i \mu_i=0,$$ where $\beta_0,\beta_1,...,\beta_m$ are
constants such that $(\beta_0,\beta_1,...,\beta_m)\neq (0,0,...,0).$
Since the set $\{\mu_1,\mu_2,...,\mu_m\}$ is linearly independent,
then $\beta_0\neq 0$ and hence, we can write
\begin{equation} \label{gl}g=\sum_{i=1}^{m}a_i \mu _i \text{ where } a_i\in K.
\end{equation}
If we take $y=e$ in (\ref{nf}) we get \begin{equation} \label{nfe}
f=h(e)g+\sum_{i=1}^{m}h_{J(i)}(e)b_i\mu_i.\end{equation} Replacing
$g$ by its form (\ref{gl}) in (\ref{nf}), and using (\ref{nfe}) we
get
$$h(e)\sum_{i=1}^{m}a_i \mu _i(xy)+
\sum_{i=1}^{m}h_{J(i)}(e)b_i \mu_i(xy)=\sum_{i=1}^{m}a_i \mu
_i(x)h(y)+\sum_{i=1}^{m}b_i \mu _i(x)h_{J(i)}(y).$$ This can be
written as follows
$$\sum_{i=1}^{m}\mu _i(x)\left[a_i\left(h(e)\mu_i(y)-h(y)\right)+b_i \left(h_{J(i)}(e)\mu_i(y)-h_{J(i)}(y)\right)\right]=0.$$
 Since the set
$\{\mu_1,\mu_2,...,\mu_m \}$ is linearly independent we deduce that
\begin{equation} \label{kjn}[a_ih(e)+b_ih_{J(i)}(e)]\mu_i=a_ih+b_ih_{J(i)} \text{ for all } i=1,2,...,m.\end{equation}
Here we discuss two cases\\
\verb"Case 1" : Suppose that there exists $q \in \{1,...,n\}$ such
that $a_kb_s\neq a_sb_k$ for some pair $k,s \in
\{n_{q},...,m_{q}\}.$\\
Replacing $i$ in (\ref{kjn}) by $k$ and $s$ and using that
$J(s)=J(k)=q$ we get respectively
\begin{equation} \label{pln}a_{k}h+b_{k}h_q=[a_{k}h(e)+b_{k}h_q(e)]\mu_{k},\end{equation}
\begin{equation} \label{plsn}a_{s}h+b_{s}h_q=[a_{s}h(e)+b_sh_q(e)]\mu_{s}\end{equation}
Since $n\geq 2$, the set $\{1,...,\widehat{q},..,n\}$ is not empty.
Let $j\in \{1,...,\widehat{q},..,n\}$ and let $p \in
\{n_j,,...,m_j\}$. If we replace $i$ in (\ref{kjn}) by $p$ and use
that $J(p)=j$ we get that
\begin{equation}[a_ph(e)+b_ph_j(e)]\mu_p=a_ph+b_ph_j.\end{equation} This three equations can be written in
a matrix form:
$$AH=B\mu,$$ where
\begin{equation} \label{matrixAn}A=\begin{pmatrix}a_k
 &b_k&0\\
a_s&b_s&0\\a_p&0&b_p
\end{pmatrix},\end{equation}
\begin{equation} \label{ht}H^{T}= (h,h_q,h_j),\end{equation} \begin{equation} \label{muu} \mu^{T}=(\mu_k,
 \mu_s,
\mu_p),\end{equation} and  \begin{equation} B=\begin{pmatrix}a_kh(e)+b_kh_q(e) &0&0\\
0&a_sh(e)+b_sh_q(e)&0\\0&0&a_ph(e)+b_ph_j(e)
\end{pmatrix}.\end{equation}
Since $a_kb_s\neq a_sb_k$ and $b_p\neq 0$ we have
$$det(A)=b_p (a_kb_s-a_sb_
k)\neq 0.$$ Hence, the matrix $A$ has an inverse:

\begin{equation} \label{comn}
A^{-1}=\frac{1}{\alpha}\begin{pmatrix} b_s&-b_k&0\\
-a_s&a_k&0\\
-b_s\frac{a_p}{b_p}&b_k\frac{a_p}{b_p}&\frac{\alpha}{b_p}
\end{pmatrix},\end{equation} where $\alpha=a_kb_s-a_sb_k$.\\
We get after some computations that there exist $a,b \in K$ such
that
\begin{equation} \label{AAn}
h=a\mu_k+b\mu_s,\end{equation}
\begin{equation} \label{A0n}
h_q=-\frac{a_s}{b_s}a\mu_k-\frac{a_k}{b_k}b\mu_s,
\end{equation}
\begin{equation} \label{A1sn}h_j=-\frac{a_p}{b_p}a\mu_k-\frac{a_p}{b_p}b\mu_s+\alpha_j
\mu_p,\end{equation}
 where $\alpha_j \in K$ for all $j\in \{1,...,\widehat{q},...,n\}$. If we put $d:=a_s/b_s$, $c:=a_k/b_k$, $c_j:=a_p/b_p$ we get that
\begin{equation} \label{h123n}
h=a\mu_k+b\mu_s,\; hq=-da\mu_k-cb\mu_s,\end{equation}
\begin{equation} \label{hjn}h_j=-c_ja\mu_k-c_jb\mu_s+\alpha_j
\mu_p,\end{equation} for all $j\in \{1,...,\widehat{q},...,n\}$.
Since $a_sb_k\neq a_kb_s$, we have $d\neq c$.\\ If $a=b=\alpha_j=0$
for all $j\in \{1,...,\widehat{q},...,n\}$ then $h=h_q=h_j=0$  which
gives that $f=0$. Since $f$ is a nonzero function we deduce that
$(a,b,\alpha_1,...,\widehat{\alpha_q},..,\alpha_n )\neq
(0,0,0,...,0)$.\\Since we have $m_j-n_j\geq 1$ for all $j\in
\{1,...,\widehat{q},...,n\}$, then the set
$\{n_j,...,\widehat{p},...,m_j\}$ is not empty. Hence we can choose
$l\in \{n_j,...,\widehat{p},...,m_j\}$, where $j\in
\{1,...,\widehat{q},...,n\}$. Replacing $i$ in equation (\ref{kjn})
by $l$ we get
\begin{equation} \label{ijkpn} [a_lh(e)+b_lh_j(e)]\mu_l=a_lh+b_lh_j.\end{equation} Deriving the
expressions of $h(e)$ and $h_j(e)$ from equations (\ref{h123n}) and
(\ref{hjn}) and replacing the functions $h$, and $h_j$ by their
forms, the equation (\ref{ijkpn}) can be written as follows
$$[a_l(a+b)-b_l(c_ja+c_jb-\alpha_j)]\mu_l=a_l(a\mu_k+b\mu_s)+b_{l}(-c_ja\mu_k-c_jb\mu_s+\alpha_j \mu_p).$$
Hence
$$[(a+b)(a_l-b_lc_j)+b_l \alpha_j]\mu_l=a(a_l-b_lc_j)\mu_k+b(a_l-b_lc_j)\mu_s+b_l\alpha_j \mu_p.$$
The set $\{\mu_s,\mu_k,\mu_p,\mu_l\}$ is linearly independent
because $\mu_s,\mu_k,\mu_p,\mu_l$ are distinct.\\
Thus we get that
\begin{equation} \label{coefn}b_l\alpha_j=0,\;a(a_l-b_lc_j)=0,
\;b(a_l-b_lc_j)=0,\;(a+b)(a_l-b_lc_j)+b_l \alpha_j=0.\end{equation}
Since the coefficients $b_i$ are nonzero for all $i=1,2,...,m$ we
get from the first relation of (\ref{coefn}) that $\alpha_j =0$ for
all $j\in \{1,..,\widehat{q},...,n\}$. Hence $(a,b)\neq (0,0)$
because $(a,b,\alpha_1,..,\alpha_n)\neq (0,0,0,..,0)$. Thus we get
from the second and the third equation of (\ref{coefn}) that
$a_l=c_jb_l$ for all $l \in \{n_j,...,\widehat{p},...,m_j\}$. Since
$c_j=a_p/b_p$ we get that $a_i=c_jb_i$ for all $i \in
\{n_j,...,m_j\}$ where $j\in \{1,2,...,\widehat{q},...,n\}$. This
means that $a_p/b_p=a_l/b_l=c_j$, and hence $a_pb_l=a_lb_p$ for all
$p,l\in \{n_j,...,m_j\}$, and for all $j\in
\{1,2,...,\widehat{q},...,n\}$. This proves the uniqueness of $q \in
\{1,2,...,n\}$.
\\ Consequently, using the form of $g$ we obtain
\begin{align*}g&=\sum_{i=1}^{m}a_i\mu_i
=\sum_{i=1}^{m_1}a_i\mu_i+\sum_{i=n_2}^{m_2}a_i\mu_i+...+\sum_{i=n_n}^{m}a_i\mu_i\\
&=\sum_{i=n_q}^{m_q}a_i\mu_i+\sum_{j\neq
q}^{n}\sum_{i=n_j}^{m_j}a_i\mu_i\\&=\sum_{i=n_q}^{m_q}a_i\mu_i+
\sum_{j\neq q}^{n}c_j\sum_{i=n_j}^{m_j}b_i\mu_i. \end{align*}Hence
\begin{equation} \label{gform} g =\sum_{i=n_q}^{m_q}a_i\mu_i+\sum_{j\neq q}^{n}c_j
g_j.\end{equation}
Here, we treat two cases:\\
\verb"Case 1.1":  Suppose that $m_q-n_q=1$. We may take $s=n_q$ and
$k=n_q+1=m_q$. In this case we have
$$g=a_{n_q}\mu_{n_q}+a_{m_q}\mu_{m_q}+\sum_{j\neq q}^{n}c_jg_j, \;g_q=b_{n_q}\mu_{n_q}+b_{m_q}\mu_{m_q},$$ and we get from
(\ref{h123n}) that $$h(e)=a+b,\;h_q(e)=-da-cb,\;
h_j(e)=-c_ja-c_jb.$$ Using those forms we obtain that
\begin{align*}f&=h(e)g+h_q(e)g_q+\sum_{j\neq q}^{n}h_j(e)g_j\\
&=h(e)(a_{n_q}\mu_{n_q}+a_{m_q}\mu_{m_q})+h_q(e)g_q+\sum_{j\neq
q}^{n}(h_j(e)+c_jh(e))g_j\\
&=(a+b)(a_{n_q}\mu_{n_q}+a_{m_q}\mu_{m_q})+(-da-cb)(b_{n_q}
\mu_{n_q}+b_{m_q}\mu_{m_q}).\end{align*} Since $a_{n_q}=db_{n_q}$
and $a_{m_q}=cb_{m_q}$ we obtain after simplification that
\begin{align*}f&=b(d-c)b_{n_q}\mu_{n_q}+a(c-d)b_{m_q}\mu_{m_q}\\&=(d-c)(bb_{n_q}\mu_{n_q}-ab_{m_q}\mu_{m_q}).\end{align*}\\
\verb"Case 1.2" :  Suppose that $m_q-n_q\geq 2$. In this case we can
choose an element $r\in \{n_q,...,m_q\}$ such that $r\neq s$ and
$r\neq k$. Substituting the forms of $h$, $h_q$, $h(e)$, and
$h_q(e)$ in (\ref{kjn}) and replacing $i$ by $r$ we find
\begin{equation}\label{eqn}[a_r(a+b)+b_r(-da-cb)]\mu_r=a_r(a\mu_k+b\mu_s)+b_r(-da\mu_k-cb\mu_s),\end{equation}
which can be written as follows
$$[a_r(a+b)+b_r(-da-cb)]\mu_r=a(a_r-b_rd)\mu_k+b(a_r-cb_r)\mu_{s_q}.$$
Since $\mu_s,\mu_k,\mu_r$ are distinct the set
$\{\mu_r,\mu_k,\mu_r\}$ is linearly independent. This implies that
$$a_r(a+b)+b_r(-da-cb)=0,\;a(a_r-db_r)=0,\;b(a_r-cb_r)=0.$$
Taking in account that $(a,b)\neq (0,0)$ and $c\neq d$ we deduce
that either $a\neq 0$ and $b=0$, or $b\neq 0$ and $a=0$.\\ Assume
that $a\neq 0$ and $b=0$. In this case we find that $a_r=db_r$.
Noting that $a_s=db_s$ we deduce that $a_i=db_i$ for all $i\in
\{n_q,...,\widehat{k},...,m_q\}$. Thus, using equation (\ref{gform})
we get that
\begin{align*}g&=\sum_{i=n_q}^{m_q}a_i\mu_i+\sum_{j=1,j\neq q}^{n}c_j
g_j\\
&=d\sum_{i=n_q,i\neq k}^{m_q}b_i\mu_i+a_k\mu_k+\sum_{j\neq
q}^{n}c_j g_j\\
&=d\sum_{i=n_q,i\neq k}^{m_q}b_i\mu_i+cb_k\mu_k+\sum_{j\neq
q}^{n}c_j g_j\\ &=dg_q+(c-d)b_k\mu_k+\sum_{j\neq q}^{n}c_j
g_j.\end{align*} Using that $b=0$ we deduce from (\ref{h123n}) that
$h=a\mu_k,\;h_q=-da\mu_k,\;h_j=-c_ja\mu_k,$ for all $j\in
\{1,...,\widehat{q},...n\}$,
\begin{align*}f&=h(e)g+h_qg_q+\sum_{j=1,j\neq q}^{n}h_j(e)g_j\\
&=ag-d a g_k-a\sum_{j=1,j\neq q}^{n}c_j g_j\\
&=a\left(dg+(c-d)b_k\mu_k+\sum_{j=1,j\neq q}^{n}c_j
g_j\right)-dag-a\sum_{j=1,j\neq q}^{n}c_j g_j\\
&=a(c-d)b_k\mu_k.\end{align*} So we obtain solution (a).\\ The case
of $a=0$ and $b\neq 0$ can
be treated in a similar way.\\
\verb" Case 2" : Suppose that $a_sb_k=a_kb_s$ for all $k,s \in
\{n_j,...,m_j\}$ and all $j=1,2,...,n$. This means that there exists
$v_j \in K $ such that $a_i=v_jb_i$ for all $i\in  \{
n_j,...,m_j\}$. Hence we have
$$g=\sum_{i=1}^{m}a_i\mu_i=\sum_{i=1}^{m_1}a_i\mu_i+\sum_{i=n_2}^{m_2}a_i\mu_i+...+\sum_{i=n_n}^{m}a_i\mu_i$$
$$=v_1\sum_{i=1}^{m_1}b_i\mu_i+v_2\sum_{i=n_2}^{m_2}b_i\mu_i+...+v_n\sum_{i=n_n}^{m}b_i\mu_i=\sum_{j=1}^{n}v_jg_j.$$ Substituting the last form of $g$ in
(\ref{wx0n}) we get that
\begin{align*}f(xy)&=g(x)h(y)+\sum_{j=1}^{n}g_j(x)h_j(y)\\
&=\sum_{j=1}^{n}v_jg_j(x)h(y)+\sum_{j=1}^{n}g_j(x)h_j(y)\\
&=\sum_{j=1}^{n}g_j(x)(v_jh+h_j)(y).\end{align*} This equation has
the form of the functional equation (\ref{gG0n}). According to Lemma
\ref{f0n} we deduce that $f=0$, which is not true by hypothesis.
Thus this case cannot occur. This completes the proof.
\end{proof}
Theorem \ref{soln} gives all solutions of the functional equation
(\ref{gmn}).
\begin{thm} \label{soln} Let $(f,g,h,h_1,h_2,...,h_n)$ be a solution of
the functional equation (\ref{gmn}). Then the solutions have one of
the following forms:
\begin{enumerate} \item $$f=ab\chi, g=b\chi-\sum_{k=1}^{n}c_k
g_k,h=a\chi,h_j=ac_j\chi, \text{ for }j=1,2,...,n,$$ where $\chi:M
\longrightarrow K$ is any nonzero multiplicative function such that
$\chi \neq \mu_i$ for all $i \in \{1,2,...,m\}$, and $a,b \in
K^{\ast}$, and $c_j \in K$ for $j =1,2,...,n$.

\item there exists a unique $q\in \{1,..,n\}$ and there exist constants
$a,b,c,d,c_j\in K$ for $j\in \{1,..,\widehat{q},...,n\}$, with
$c\neq d$ and $(a,b)\neq (0,0)$ such that  \begin{itemize} \item If
$m_q-n_q=1$, then
\begin{equation}f=(d-c)(bb_{n_q}\mu_{n_q}-ab_{m_q}\mu_{m_q}),\end{equation}
\begin{equation}g=db_{n_q}\mu_{n_q}+cb_{m_q}\mu_{m_q}+\sum_{j=1,j\neq q}^{n}c_j
g_j, \end{equation} \begin{equation}h=a\mu_{m_q}+b\mu_{n_q},\;
h_q=-da\mu_{m_q}-cb\mu_{n_q},\end{equation}\begin{equation}h_j=-c_ja\mu_{m_q}-c_jb\mu_{n_q}
\text{ for all } j \in
\{1,..,\widehat{q},...,n\}.\end{equation}\item If $m_q-n_q \geq2$,
then \begin{equation}f=a(c-d)b_{k}\mu_{k},
\;g=dg_{q}+(c-d)b_{k}\mu_{k}+\sum_{j=1,j\neq q}^{n}c_j
g_j,\end{equation}
\begin{equation}h=a\mu_{k},\;h_{q}=-da\mu_{k},\;h_j=-c_j\mu_{k} \text{ for all
}j\in \{1,..,\widehat{q},...,n\},\end{equation} where $k\in
\{n_q,...,m_q\}$ and $a\neq 0$.\end{itemize}

\item
$f=0,\;g=\sum_{k=1}^{n}c_k g_k$, $h:M\longrightarrow K$ is any
nonzero function, $h_j=-c_jh$, for all $j\in \{1,..,n\}$ where
$c,c_j\in K.$

\item $f=0$, $g:M\longrightarrow K$ is any function, $h=h_j=0$ for
all $j\in \{1,..,n\}$. \end{enumerate}
\end{thm}
\begin{proof} The statements (1),(2) and (3) are proved in Proposition \ref{depn} and Proposition
\ref{indpn}. It is left to treat the case of $f=0$.\\Assume that
$f=0$. The functional equation (\ref{gmn}) reduces to
\begin{equation}
\label{g1h1n}g(x)h(y)+\sum_{j=1}^{n}g_j(x)h_j(y)=0.\end{equation}
Here we discuss two cases:\\ If $h=0$ then we get that
$$\sum_{j=1}^{n}g_j(x)h_j(y)=0.$$ Using the form (\ref{fgn}) of $g$ this can be written as follows
$$\sum_{i=1}^{m}b_ih_{J(i)}(y)\mu_i(x)=0$$
Since $\mu_1,...,\mu_m$ are distinct multiplicative functions, the
set $\{ \mu_1,...,\mu_m\}$ is linearly independent, hence
$h_{J(i)}=0$ because $b_i \neq 0$ for all $i=1,...,m$. Using the
fact that $J$ is surjective we deduce that $h_j=0$ for all $j\in
\{1,2,...,n\}$. This is the solution (5) of Theorem \ref{soln}.\\If
$h\neq 0,$ then there exists $y_0\in K$ such that $h(y_0)\neq 0$.
Taking $y=y_0$ in (\ref{g1h1n}) we deduce that
\begin{equation} \label{q5}g=\sum_{j=1}^{n}c_jg_j,\end{equation} where
$c_j=-h_j(y_0)/h(y_0)$. Replacing $g$ in (\ref{g1h1n}) by its new
form (\ref{q5}) we get $$\sum_{j=1}^{n}[c_jh(y)+h_j(y)]g_j(x)=0.$$
Proceeding as in the first case we deduce that $h_j=-c_jh$  for all
$j\in \{1,2,...,n\}$. This is the solution (4), and this completes
the proof.
\end{proof}
The following theorem gives the solutions of the functional equation
(\ref{gmn}) when $n=2$.
\begin{thm} \label{sol} Let $(f,g,h,h_1,h_2)$ be a solution of
the functional equation
\begin{equation} \label{n3}f(xy)=g(x)h(y)+g_1(x)h_1(y)+g_2(x)h_2(y),\;x,y \in M,\end{equation} such that
$g_1=\sum_{i=1}^{N}b_i\mu_i$, $g_2=\sum_{i=N+1}^{m}b_i\mu_i$, where
$N \geq 2$, $m-N\geq 2$, $\{\mu_1,...,\mu_m\}$ are distinct nonzero
multiplicative functions and $b_1,...,b_m$ are all nonzero.\\Then we
have one of the following forms:
\begin{enumerate} \item $$f=ab\chi,
g=b\chi-c_1g_1-c_2g_2,\;h=a\chi,h_1=ac_1\chi,h_2=ac_2\chi,$$ where
$\chi:M \longrightarrow K$ is any nonzero multiplicative function
such that $\chi \neq \mu_i$ for all $i \in \{1,2,...,m\}$, and $a,b
\in K^{\ast}$, and $c_1,c_2 \in K.$
 \item
$$f=(d-c)(bb_1\mu_1-ab_2\mu_2),\;g=db_1\mu_1+cb_2\mu_2+c_2g_2,$$
$$h=a\mu_{2}+b\mu_{1},\;
h_1=-da\mu_2-cb\mu_1,\;h_2=-c_2a\mu_2-c_2b\mu_1,$$ where
$a,b,c,d,c_2\in K$ such that $c\neq d$ and $(a,b)\neq (0,0)$.
\item
$$f=a(c-d)b_k\mu_k,\;g=dg_1+(c-d)b_k\mu_k+c_2g_2,\;$$
$$h=a\mu_k,\;h_1=-da\mu_k,\;h_2=-c_2\mu_k,$$ where $k\in \{1,...,N\}$, $N\geq 3$ and $a,b,c,d,c_2\in K$ such that
$c\neq d$ and $a\neq 0$.
\item
$$f=(d-c)(bb_{N+1}\mu_{N+1}-ab_{m}\mu_{m}),\;g=db_{N+1}\mu_{N+1}+cb_{m}\mu_{m}+c_1g_1,$$
$$h=a\mu_{m}+b\mu_{N+1},\;h_1=-c_1a\mu_{m}-c_1b\mu_{N+1},\;
h_2=-da\mu_{m}-cb\mu_{N+1},$$ where $m-N=2$, $a,b,c,d,c_1\in K$ such
that $c\neq d$ and $(a,b)\neq (0,0)$.
\item
$$f=a(c-d)b_k\mu_k,\;g=dg_2+(c-d)b_k\mu_k+c_1g_1,\;$$
$$h=a\mu_k,\;h_1=-c_1\mu_k\;h_2=-da\mu_k,$$ where
 $m-N\geq 3$, $k\in \{N+1,...,m\}$ and
$a,b,c,d,c_1\in K$ such that $c\neq d$ and $a\neq 0$.
\item
$f=0,\;g=c_1g_1+c_2g_2$, $h:M\longrightarrow K$ is any nonzero
function, $h_1=-c_1h,\;h_2=-c_2h$, where $c_1,c_2\in K.$
\item
$f=0$, $g:M\longrightarrow K$ is any function, $h=h_1=h_2=0$.
\end{enumerate}
\end{thm}

% ----------------------------------------------------------------
\bibliographystyle{amsplain}
\vspace{1cm}
Authors'address. University Ibn Zohr, Faculty of Sciences,
Department of Mathematics, Agadir, Morocco.
\email{ keltoumabelfakih@gmail.com; elqorachi@hotmail.com}%
\end{document}